\documentclass[reqno,11pt]{amsart}
\usepackage[latin1]{inputenc}
\usepackage[T1]{fontenc}
\usepackage{lmodern}
\usepackage{a4wide}
\usepackage{amsmath,amsfonts,amsthm,amssymb}
\usepackage{hyperref}
\usepackage{mathtools}
\usepackage{color}
\usepackage{dsfont}
\usepackage{pgf,tikz}
\usepackage{mathrsfs}
\usepackage{enumerate}
\usepackage[shortlabels]{enumitem}
\everymath{\displaystyle}
\usetikzlibrary{arrows}

\theoremstyle{plain}
\newtheorem{theorem}{\bf Theorem}[section]
\newtheorem{corollary}{\bf Corollary}[section]
\newtheorem{lemma}{\bf Lemma}[section]
\newtheorem{remark}{\bf Remark}[section]
\newtheorem{definition}{\bf Definition}[section]

\newtheorem{example}{\bf Example}[section]

\numberwithin{equation}{section}

\title[Some approximation results in time and space dependent Musielak spaces]{Some approximation results in time and space dependent Musielak spaces}

\author[]{Youssef Ahmida}
\address{Youssef Ahmida\\
Choua\"ib Doukkali University, Higher School of Education and Training of El Jadida,
Sciences and Technologies Team (ESTE), Road Azzemour, El Jadida, Morocco}
\email{youssef.ahmida@gmail.com}

\author[]{Ahmed Youssfi}
\address{Ahmed Youssfi\\
Sidi Mohamed Ben Abdellah University, National School of Applied Sciences,
Laboratory of Mathematical Analysis and Applications,
My Abdellah Avenue, Road Imouzer, P.O. Box 72 Fès-Principale, Fez, 30 000, Morocco}
\email{address: ahmed.youssfi@usmba.ac.ma ; ahmed.youssfi@gmail.com}

\begin{document}

\maketitle
\begin{abstract}
We provide density results for smooth functions in non-reflexive Musielak spaces defined up on time- and space- dependent modular functions. These Musielak spaces encompass a broad class of functional framework such as Bochner spaces, inhomogeneous Orlicz spaces as well as variable-exponent Sobolev spaces.
\end{abstract}
{\small {\bf Key words and phrases:}  Musielak spaces, Time and space dependent Musielak functions, Density}

{\small{\bf Mathematics Subject Classification (2020)}:  46E35, 11B05}

\tableofcontents
\section{Introduction}
It is in our purpose to provide density results for a relevant  functional framework with time and space dependent inhomogeneity, not necessarily reflexive nor separable, that can be involved in the theory of existence of variational evolution problems governed by operators having fast or slowly growing coefficients.

\par It is well known that the framework of Bochner spaces is not proper to study nonhomogeneous structure including the $p(\cdot)$-Laplacian or nonstandard Orlicz spaces since the analogous expression of the norm in classical Bochner spaces defined by 
$$
	\|u\|_{L^p(0,T;W^{1,p}_0(\Omega))}:=\left(\int_0^T \|\nabla u\|^p_{L^p(\Omega)}dt\right)^{\frac{1}{p}}
$$
 has no meaning in general (see  (\ref{luxem_norm}) below). Furthermore, in view of  \cite[Proposition 1.3 p. 218 and Remarks p. 219]{Donaldson74}, the analogous version of the identity 
$$
L^p\left(I,L^p(\Omega\right))=L^p(\Omega_T)
$$
in  Orlicz spaces holds only under very strong conditions which impose on the corresponding Orlicz functions of being equivalent to certain power $p$ with $1\leqslant p\leqslant\infty$. Consequently, the resulting Orlicz spaces become  reflexive and separable.
\par In the scale of Orlicz spaces, this problem is freed by Donaldson \cite{Donaldson74} who have 
introduced the idea of inhomogeneous Orlicz-Sobolev spaces $W^{1,x}L_\varphi(\Omega_T)$, defined as the class of all measurable functions $u : \Omega_T\to\mathbb{R}$ such that $u\in L_\varphi(\Omega_T)$ and $\nabla_x u\in L_\varphi(\Omega_T)$ where $\nabla_x$ stands for the distributional gradient with respect to the spatial variable. 
This idea was then used in \cite{EM05} to study variational initial-boundary value problems in the Orlicz spaces setting.
\par In the setting of variable exponent Lebesgue spaces, unsteady problems with time -and space- dependent nonlinearity were considered first in \cite{AS07,AS09}. The authors solved these problems in a functional framework quite different from Bochner spaces. Then, using the theory of monotone operators, the authors of \cite{DNR12,nagele_thesis} studied some parabolic equations governed by an operator having a $p(t,x)$-structure. Under a globally $\log$-H\"older continuity on the variable exponent, they developed basic properties and solved the problem in the energy space 
$$
	X(\Omega_T):=\left\{u\in L^2(\Omega_T)^N\big| \nabla u\in L^{p(\cdot,\cdot)}(\Omega_T)^{N\times N}, 
	u(\tau,\cdot)\in V_\tau(\Omega) \right\}
$$
where
$$
V_\tau(\Omega):=\left\{ v\in L^2(\Omega)^N \cap W_0^{1,1}(\Omega)^N, \nabla v\in L^{p(\tau,\cdot)}(\Omega)^{N\times N} \right\}
$$
We point out that the structure of $X(\Omega_T)$ reduces to Bochner spaces when the nonlinearity $p(\cdot,\cdot)$ is a constant function as one can see by using \cite[Lemma 4.2]{DNR12}. Note that the presence of $L^2$ in the energy space is dictated by problem study considerations. 
\par Coming up to the scale of Musielak spaces, we provide new functional spaces constructed from Musielak functions depending on time and space, in which Poincaré's inequality is not necessarily satisfied. The main challenge is to prove the density of smooth functions in these spaces.  Recall that smooth functions are not dense in general in Musielak spaces, in particular in variable exponent Sobolev spaces. Thus a condition of regularity must be imposed on the variable exponent.
It is known that a sufficient condition to get the density in variable exponent Sobolev spaces is the so called $\log$-H\"older condition (cf. \cite{bookCF,DHHR}).\\
We obtain approximation results for functional spaces built from   Musielak functions satisfying a regular condition (see \cite{AYGS2017,AFY_ms19,AY2018.art3}).
\par The manuscript is organized as follows. In section  \ref{background_par} we provide some definitions and notations of Musielak spaces. In Section \ref{Sec_ts_space} we introduce the time and space dependent Musielak spaces. Section \ref{Sec_app space} is devoted to approximation results of smooth functions in these spaces. In section \ref{Sec_app time}, we provide an energy parabolic space in which we prove that smooth functions are dense.
\section{Basic properties of Musielak spaces}\label{background_par}
In this section, we summarize some basic facts on Musielak spaces.  More details can be found, for instance, in  \cite{ahm_thesis,AYGS2017,AY2018.art3, Kaminska15,book_musielak,AY2018}. 
\begin{definition}
	A function $M$ : $\Omega\times\mathbb{R}^{+}\to \mathbb{R}^{+}$ is called $\Phi$-function, written $M\in\Phi$, if $M(\cdot,s)$ is a measurable function for every $s\geq 0$ and $M(x,\cdot)$ is a convex function for almost every $x\in\Omega$ with $M(x,0)=0$, $M(x,s)>0$ for $s>0$, $M(x,s)\rightarrow\infty \mbox{ as }s\rightarrow\infty$ and
	$$
	\lim_{s\to 0}\frac{M(x,s)}{s}=0\quad \mbox{ and }\quad \lim_{s\to \infty} \frac{M(x,s)}{s}=\infty.
	$$
\end{definition}
Let $\Psi\in \Phi(\Omega)$.
We denote by $\Psi^{\ast}: \Omega\times\mathbb{R}^{+}\to\mathbb{R}^{+}$ the complementary function to $\Psi$ in the sense of Young defined by
$$
\Psi^{\ast}\left(x,\xi_1\right)=\sup_{\xi_2\geq 0}\left\{\xi_1\xi_2-\Psi(x,\xi_2)\right\} \mbox{ for all } \xi_1\geq 0\mbox{ and a.e } x\in \Omega.
$$
Thus, we obtain the following Young inequality
\begin{equation}\label{younginequality}
	\xi_1\xi_2\leq \Psi(x,\xi_1)+ \Psi^\ast(x,\xi_2),\quad \forall \xi_1,\xi_2\geq 0 \mbox{ and a.e } x\in\Omega.
\end{equation}
\begin{definition}
	We say that a $\Phi(\Omega)$-function $\Psi$ is locally integrable, if for any constant number $c>0$ and for any compact subset $\Omega^\prime$ of  $\Omega$ we have
	\begin{equation}\label{incBM}
		\int_{\Omega^\prime} \Psi(x,c)dx<\infty.\tag{$\mathcal{L}^1_{loc}$}
	\end{equation}
\end{definition}
\begin{remark}
	Note here that (\ref{incBM}) holds if for instance $\Psi$ satisfies (\ref{X1}) (cf. \cite{AFY_ms19}, \cite[Remark 3.1]{AY_var19}).
\end{remark}
We define the Musielak space, denoted $L_\Psi(\Omega)$ (resp. $E_\Psi(\Omega)$), as the set of all measurable functions $u:\Omega\rightarrow \mathbb{R}$, such that 
$
\rho_\Psi(u):=\int_\Omega \Psi(x,|u(x)|/\lambda)dx<+\infty \mbox{ for some }\lambda>0, \; (\mbox{resp. for all } \lambda>0).
$
The space $L_\Psi(\Omega)$ (resp. $E_\Psi(\Omega)$) is endowed with the so-called Luxemburg norm
\begin{equation}\label{luxem_norm}
	\|u\|_{L_\Psi(\Omega)}=\inf\bigg\{\lambda>0: \int_\Omega \Psi(x,|u(x)|/\lambda)dx\leq 1\bigg\}
\end{equation}
which makes it a Banach space. In addition, by (\ref{younginequality}) we get the following H\"{o}lder inequality
$$
\int_\Omega |\xi_1\cdot\xi_2| dx \leq2\|\xi_1\|_{L_\Psi(\Omega)}\|\xi_2\|_{L_{\Psi^{\ast}}(\Omega)}, 
$$
for all $(\xi_1,\xi_2)\in L_\Psi(\Omega)\times L_{\Psi^{\ast}}\Omega)$.
\begin{definition}[Modular convergence]
	Let $\Psi\in\Phi(\Omega)$. We say that a sequence $\{v_n\}_n$ converges to $v$ in the modular sense in $L_\Psi(\Omega)$, if there is a constant number $\delta>0$ such that $\int_\Omega \Psi\left(x,\frac{|v_n-v|}{\delta}\right)dx\to 0 \mbox{ as } n\to\infty$.
\end{definition}
In what follows we us the following notations.
\begin{itemize}
	\item $\mathcal{C}^\infty_0(\Omega_T)$ the set of infinitely differentiable functions compactly supported in $\Omega_T$.
	\item $\mathcal{C}^\infty_0(\overline{\Omega}_T)$  the set of the restriction to $\Omega_T$ of the functions belonging to $\mathcal{C}^\infty_0(\mathbb{R}^{N+1})$.
	\item $\mathcal{C}^\infty_0([0,T];\mathcal{C}^\infty_0(\Omega ))$ the set of the restriction to $(0,T)$ of the functions belonging to the set $\mathcal{C}^\infty_0(\mathbb{R};\mathcal{C}^\infty_0(\Omega))$.
\end{itemize}
The following result ensures that the functions in $L_\Psi(\Omega)$ have distributional derivatives.
\begin{lemma}\label{lem_emb_L1}
	Let $\Omega$ be an open subset of $\mathbb{R}^N$ and  let $(\Psi,\Psi^\ast)$ be a pair of complementary $\Phi(\Omega)$-functions.  Then the following continuous embedding	
	$
	L_{\Psi}(\Omega) \hookrightarrow  L^1_{loc}(\Omega)
	$
	holds if one of the following assumptions holds
	\begin{itemize}
		\item [$1)$] $\Psi^\ast$ satisfies (\ref{incBM}) 
		\item [$2)$]  $\Psi$ satisfies
		\begin{equation}\label{cond_essinf}
			\underset{{x\in\Omega}}{\mathrm{ess\,inf\,}}\Psi(x,1)\geq c>0.
		\end{equation}
	\end{itemize} 
\end{lemma}
\begin{proof} 
\item 1) Assume first that $\Psi^\ast$ satisfies (\ref{incBM}). Let $u\in L_\Psi(\Omega)$ and let $K$ be a compact subset of $\Omega$. By H\"{o}lder's inequality we get
$$
\int_K |u(x)|dx
\leq 2\|u\|_{L_{\Psi}(\Omega)} \|1\|_{L_{\Psi^\ast(K)}} 
\leq 2\|u\|_{L_{\Psi}(\Omega)}\left(\int_{K}\Psi^\ast(x,1)dx+1\right)<\infty,
$$	
which shows that $u\in L^1_{loc}(\Omega)$.
\item 2) Assume now that (\ref{cond_essinf}) holds. Let $u\in L_\Psi(\Omega)$ and consider the set  
		$$
		\Omega^\prime=\left\{x\in K,\; \frac{|u(x)|}{\|u\|_{L_{\Psi}(\Omega)}}\geq1\right\},
		$$
		where $K$ is a compact set of $\Omega$. Using (\ref{cond_essinf}) we can write
		$
		s\leq \frac{1}{c}\Psi(x,s) \mbox{ for all } s\geq 1.
		$
		Hence, we get
$$
\begin{array}{lll}\displaystyle
\int_K\frac{|u(x)|}{\|u\|_{L_{\Psi}(\Omega)}}dx&=\displaystyle \int_{\Omega^\prime}\frac{|u(x)|}{\|u\|_{L_{\Psi}(\Omega)}}dx+\int_{K\setminus \Omega^\prime} \frac{|u(x)|}{\|u\|_{L_{\Psi}(\Omega)}}dx\\
&\leq\displaystyle\frac{1}{c} \int_{\Omega^\prime}\Psi\left(x,\frac{|u(x)|}{\|u\|_{L_{\Psi}(\Omega)}}\right)dx+|K|\leq\displaystyle\frac{1}{c}+|K|.
\end{array}
$$
Thus, we obtain
		$
		\|u\|_{L^1(K)}\leq \left(\frac{1}{c}+|K|\right)\|u\|_{L_{\Psi}(\Omega)}.
		$
\end{proof}
\par Let $\alpha:=(\alpha_1,\cdots,\alpha_N)$ be multi-indices with $|\alpha|=\alpha_1+\cdots+\alpha_N$. In the sequel, we denote by $D^\alpha$ the distributional derivative with respect to the space variable. If $\Psi^\ast$ satisfies (\ref{incBM})  (or $\Psi\in \Phi$ satisfies (\ref{cond_essinf})), Lemma \ref{lem_emb_L1} ensures that the functions in $L_\Psi(\Omega)$ have distributional derivatives. Therefore, we can define the Musielak-Sobolev space  denoted $W^{1}L_M(\Omega)$, (resp. $W^{1}E_M(\Omega)$) as the set of all measurable functions $u:\Omega\rightarrow \mathbb{R}$, such that for all $|\alpha|\leq1$, the functions $D^\alpha u$ belong to $L_\Psi(\Omega)$ (resp. $E_\Psi(\Omega)$) that is if
$
\rho_\Psi(D^\alpha u)=\int_\Omega \Psi(x,|D^\alpha u(x)|/\lambda)dx<+\infty \mbox{ for some }\lambda>0, 
$
(resp. for all  $\lambda>0$). When endowed  with the norm
$$
\|u\|_{W^{1}L_\Psi(\Omega)}=\inf\left\{\lambda>0: \sum_{|\alpha|\leq1}\int_\Omega \Psi(x,|D^\alpha u(x)|/\lambda)dx\leq 1\right\}
$$
$W^{1}L_\Psi(\Omega)$ is a Banach space provided that $\Psi$ satisfies (\ref{cond_essinf}) or $\Psi^\ast$ satisfies (\ref{incBM}).
In what follows, we identify the space $W^1L_\Psi(\Omega)$ to a subspace of the product $\Pi L_\Psi$ of $(1+N)$ copies of  $L_\Psi(\Omega)$. Note that if a $\Phi(\Omega)$-function $\Psi$ satisfies (\ref{incBM}) the set of $\mathcal{C}^{\infty}_{0}(\Omega)$-functions is contained in $W^{1}L_\Psi(\Omega)$ and so we can define the space $W^{1}_0E_\Psi(\Omega)$ as the norm closure of $\mathcal{C}^{\infty}_{0}(\Omega)$ in $W^{1}L_\Psi(\Omega)$. In addition, if the pair of complementary $\Phi(\Omega)$-functions $(\Psi,\Psi^\ast)$ that satisfy booth (\ref{incBM}) the  weak-$^\ast$ topology $\sigma(\Pi L_\Psi, \Pi E_{\Psi^\ast})$ is well defined (see \cite{AY2018}) and so we can also define the space $W^{1}_0L_\Psi(\Omega)$ as the closure of $\mathcal{C}^{\infty}_{0}(\Omega)$ with respect to the weak-$^\ast$ topology $\sigma(\Pi L_\Psi, \Pi E_{\Psi^\ast})$.

Recall that if  a $\Phi(\Omega)$-function$\Psi$ satisfies (\ref{incBM}) the dual  space of $W^{1}_{0}E_{\Psi}(\Omega)$, denoted by $W^{-1}L_{\Psi^\ast}(\Omega)$, consists of all distributions $u \in (C^\infty_{ 0}(\Omega) )^{\prime}$ which can be written (see \cite[Theorem A.1.]{AY_var19})
$$
u=\Sigma_{|\alpha| \leq 1}(-1)^{|\alpha|} D^\alpha u_\alpha;\quad \text { where }  u_\alpha \in L_{\Psi^\ast}(\Omega) \mbox{ for all } |\alpha| \leq 1.
$$
\begin{definition}[Balanced $\Phi(\Omega)$-functions]
	A $\Phi(\Omega)$-function $M$ is said to satisfy the $\log$-H\"older condition if there exists a function 
	$\varrho:\big[0, {1}/{2}\big]\times\mathbb{R}^+\to\mathbb{R}^+$ such that $\varrho(\cdot,s)$ and $\varrho(x,\cdot)$ are non-decreasing functions and for all $x,y\in\overline{\Omega}$ with $|x-y|\leq\frac{1}{2}$ and for any constant $c>0$
	\begin{equation}\label{X1}
		M(x,s)\leq\varrho(|x-y|,s)M(y,s)\qquad\mbox{ with } \qquad\limsup_{\varepsilon\rightarrow0^+}
		\varrho(\varepsilon, c\varepsilon^{-N})<\infty.\tag{$\mathcal{L}_{hol}$}	
	\end{equation}
\end{definition}
We underline here that the assumption (\ref{X1}) is obviously satisfied in the framework of Orlicz spaces. Furthermore, (\ref{X1}) is a natural extension of the $\log$-H\"older condition used in the setting of variable exponent Sobolev spaces (see \cite[Example]{AFY_ms19}). In connection with the Lavrentiev phenomenon, (\ref{X1}) is a sufficient condition to avoid such a phenomena in the case of the double phase function (cf. \cite{AYGS2017,min-double-reg1}). The approximation results obtained in \cite{AYGS2017} allow to identify the weak derivatives of functions in Musielak-Sobolev spaces via smooth functions with respect to the so-called modular topology which is well suitable for the theory of existence of PDEs in nonreflexive spaces, as it is now well known in the setting of Orlicz spaces (see \cite{EM05,Gossez82,Gossez1987}), rather than the norm topology. We also recall that a $\Phi(\Omega)$-function $M$ that satisfies (\ref{X1}) is necessarily locally integrable (\ref{incBM}).
\section{Time and space dependent Musielak spaces}\label{Sec_ts_space}
We provide an appropriate functional framework of time and space dependent Musielak spaces which substitute the classical Bochner structure and includes the whole of the above inhomogeneous spaces (see Example \ref{exp_inh} below). 
\begin{definition}
	Let $M\in\Phi(\Omega_T)$. Suppose that the complementary $\Phi$-function $M^\ast$ of $M$ satisfies (\ref{incBM}) on $\Omega$, we define
	$$
	\begin{array}{c}
		W^{1,x}L_{M}(\Omega_T):=\Big\{u:\Omega_T\to \mathbb{R}: u\in L_M(\Omega_T), \nabla u\in L_M(\Omega_T)^N\Big\}\\
	\mbox{ and }\\
		W^{1,x}E_{M}(\Omega_T):=\Big\{u:\Omega_T\to \mathbb{R}: u\in E_M(\Omega_T), \nabla u\in E_M(\Omega_T)^N\Big\}.
	\end{array}
	$$
	We denote $\nabla u$ the vector gradient with respect to the space variable. These spaces  are normed by
	$
	\|u\|_{W^{1,x}L_{M}(\Omega_T)}:=\|u\|_{L_M(\Omega_T)}+
	\|\nabla u\|_{L_M(\Omega_T)^N}.
	$
\end{definition}
Observe that the space $W^{1,x}L_{M}(\Omega_T)$ can be identified to a subspace of the product  $\Pi L_M$ of $(N+1)$ copies of $L_M(\Omega_T)$.
\begin{remark}
	Let $M\in\Phi(\Omega_T)$ and suppose that the complementary $\Phi$-function $M^\ast$ of $M$ satisfies (\ref{incBM}) on $\Omega$, then by Fubini's theorem one can easily check that
	$$
	\begin{array}{c}
			W^{1,x}L_{M}(\Omega_T)=\Big\{v(\cdot,x):(0,T)\to  W^1L_{M}(\Omega): \;v\in L_M(\Omega_T), \nabla v\in L_M(\Omega_T)^N\Big\}\\
	\mbox{ and }\\
		W^{1,x}E_{M}(\Omega_T)=\Big\{v(\cdot,x):(0,T)\to  W^1E_{M}(\Omega): \;v\in E_M(\Omega_T), \nabla v\in E_M(\Omega_T)^N\Big\}.
	\end{array}
	$$
\end{remark}
\begin{definition}
	Let $(M,M^\ast)$ be a pair of complementary $\Phi(\Omega_T)$-functions satisfy booth (\ref{incBM}) on $\Omega$. We define
	$$
	\begin{array}{cc}
			W^{1,x}_{0}L_{M}(\Omega_T):=\Big\{u:(0,T)\to W^1_0L_{M}(\Omega): u\in L_M(\Omega_T), \nabla u\in L_M(\Omega_T)^N\Big\}\\
		\mbox{ and }\\
		W^{1,x}_{0}E_{M}(\Omega_T):=\Big\{u:(0,T)\to W^1_0E_{M}(\Omega): u\in E_M(\Omega_T), \nabla u\in E_M(\Omega_T)^N\Big\}.
	\end{array}
	$$
	These spaces will be equipped with the norm $\|\cdot\|_{	W^{1,x}L_{M}(\Omega_T)}$.
\end{definition}
In the following lemma, we give a completeness result. The proof is somehow almost similar to that of classical Sobolev spaces (e.g. \cite{AF}), so we have omitted it.
\begin{lemma}
	Let $M\in\Phi(\Omega_T)$ and assume that $M^\ast$ satisfies (\ref{incBM})  (or $M$ satisfies (\ref{cond_essinf})) on $\Omega_T$. Equipped with the norm $\|\cdot\|_{W^{1,x}L_{M}(\Omega_T)}$,  $W^{1,x}L_{M}(\Omega_T)$ is a Banach space.
\end{lemma}
\begin{example}\label{exp_inh}
	We give here some examples of time and space dependent Musielak-type spaces that are admissible in our investigation. 
	\item 1) If $M(t,x,s)=|s|^p$, with $1<p<\infty$ then 
	\begin{itemize}
		\item the space $W^{1,x}L_{M}(\Omega_T)$ is nothing but  the classical Bochner space $L^p(0,T;W^{1,p}(\Omega))$ (e.g \cite{AF,Lions69}) defined as 
		$$
		L^p(0,T;W^{1,p}(\Omega)):=
		\left\{u:(0,T)\to W^{1,p}(\Omega): \int_{0}^T\|u\|^p_{W^{1,p}(\Omega)}dt<\infty\right\}.
		$$
		\item since $M(t,x,s)=|s|^p$ and it's complementary satisfy booth the $\Delta_2$-condition, by Mazur's lemma the norm and the weak closure of $C^\infty_0(\Omega)$-functions in $W^{1,p}(\Omega)$ coincide. We obtain the reflexive Sobolev-Bochner space
		$$
		\begin{array}{lll}
				W^{1,x}_{0}E_{M}(\Omega_T)=	W^{1,x}_{0}L_{M}(\Omega_T)=L^p(0,T;W^{1,p}_0(\Omega)).
		\end{array}
		$$
	\end{itemize}
	\item 2) If $M(t,x,s)=\varphi(s)$ ($t$ and $x$-independent), then 
	\begin{itemize}
		\item we obtain the inhomogeneous Orlicz-Sobolev spaces (cf. \cite{Donaldson74})
		$$
		W^{1,x}L_{M}(\Omega_T)=W^{1,x}L_{\varphi}(\Omega_T),
		\quad (=W^{1,x}E_{\varphi}(\Omega_T)\mbox{ if }\varphi\in\Delta_2).
		$$ 
		\item in view of Theorem \ref{lm_appr_spt2} (see hereafter) we get the space (see \cite{MR2101516})
		$$
		\begin{array}{llll}
				W^{1,x}_{0}L_{M}(\Omega_T)=W^{1,x}_{0}L_{\varphi}(\Omega_T),
		\end{array}
		$$
		that is $W^{1,x}_{0}L_{M}(\Omega_T)$ is the closure of $\mathcal{C}^\infty_0(\Omega_T)$-functions with respect to the weak-$^\ast$ topology in $W^{1,x}L_\varphi(\Omega_T)$.
	\end{itemize}
	\item 3) If $M(t,x,s)=|s|^{p(x)}$, $1<p^-:= \underset{x\in\Omega}{\mathrm{ess\,inf}}
	\leq p(x)\leq p^+:=\underset{x\in\Omega}{\mathrm{ess\,sup}}p(x)<+\infty$, then we obtain the inhomogeneous variable exponent Sobolev space (see \cite{ALY2016})
	$$
	W^{1,x}L_{M}(\Omega_T)=W^{1,x}L^{p(x)}(\Omega_T).
	$$
\end{example}
\section{Approximations in space}\label{Sec_app space}
\begin{theorem}\label{th_appr_spt}
	Let $\Omega$ be an open subset of $\mathbb{R}^N$, $N\geq1$, satisfying the segment property and $(M,M^\ast)\in\Phi(\Omega_T)$. Assume that $M$ satisfies (\ref{X1}) on $\Omega_T$ and $M^\ast$ satisfies (\ref{incBM}) on $\Omega_T$. Let $u\in W^{1,x}L_{M}(\Omega_T)$  (resp.  $u\in 	W^{1,x}E_{M}(\Omega_T)$). Then there exists a sequence of functions $u_k\in \mathcal{C}^\infty_0(\overline{\Omega}_T)$ such that $D^\alpha  u_k\xrightarrow[k\to\infty]{} D^{\alpha} u$, $|\alpha|\leq1$, modularly (resp. in norm) in  $L_M(\Omega_T)$.
\end{theorem}
\begin{proof} 
	Setting $z:=(t,x)$, we shall show that for all $|\alpha|\leq1$ there is a constant number $\lambda>0$ such that for an arbitrary $\eta$ there exists a function $w\in\mathcal{C}^\infty_0(\overline{\Omega}_T)$ satisfying
	\begin{equation}\label{eq1_th_appr_W1X0}
		\int_{\Omega_T}M\left(z,\frac{|D^{\alpha} u(z)-D^{\alpha} w(z)|}{\lambda}\right)dz\leq \eta.
	\end{equation}
	We divide the proof into two steps.
	\par\noindent\textbf{Step 1.} In this step we will show that every function $u\in W^{1,x}L_{M}(\Omega_T)$ can be approximated modularly by a sequence of functions with compact supports in $\overline{\Omega}_T$. For that, let $\chi\in\mathcal{C}^\infty_0(\mathbb{R}^{N+1})$ be the function defined by $\chi(z)=1$  if  $|z|\leq 1$ and  $\chi(z)=0$ if $|z|\geq 2$. There exist $c\geq1$ such that 
	\begin{equation}\label{chi}
		|D^{\alpha} \chi(z)|\leq c \text{ and }|\partial_{t} \chi(z)|\leq c, \text{ for all } z\in\mathbb{R}^{N+1} \text{ and } |\alpha|\leq 1.
	\end{equation}
	For $n\geq1$ let $\chi_n(z)=\chi(z/n)$. Hence, $\chi_n(z)=1$ if  $|z|\leq n$ and $|D^\alpha\chi_n(z)|\leq \frac{c}{n^{|\alpha|}}\leq c$. Therefore, the function
	\begin{equation}\label{fnc:uR}
		u_n:=\chi_nu
	\end{equation}
	belongs to $W^{1,x}L_{M}(\Omega_T)$. Furthermore, since $u\in	W^{1,x}L_{M}(\Omega_T)$, there exists a real constant $\lambda_1>0$ such that for all $|\alpha|\leq 1$ one has $\int_{\Omega_T}M(z,|D^\alpha u(z)|/\lambda_1)dz<\infty$.	
	On the one hand one has 
	$$
	\begin{array}{lll}\displaystyle
		M\left(z,\frac{|D^\alpha u(z)-\chi_n(z)D^\alpha u(z)|}{2\lambda_1}\right) \leq M\left(z,\frac{|D^\alpha u(z)|}{\lambda_1}\right)\in L^1(\Omega_T).
	\end{array}
	$$
	On the other hand when $n\to\infty$ we have
	$$
	M\left(z,\frac{|D^\alpha u(z)-\chi_n(z)D^\alpha u(z)|}{2\lambda_1}\right)\rightarrow ,  \mbox{  a.e. in } \Omega_T.
	$$
	So that by the Lebesgue dominated convergence theorem we obtain
	\begin{equation}\label{con_seq_ur}
		\lim_{n\rightarrow\infty}\int_{\Omega_T}M\left(z,\frac{|D^\alpha u(z)-\chi_n(z)D^\alpha u(z)|}{2\lambda_1}\right)dz= 0.
	\end{equation}
	Set $\lambda_2=\max\{8c\lambda_1,4\lambda_1\}$. The convexity of $M(z,\cdot)$ allows us to get
	\begin{equation*}
		\begin{array}{lll}\displaystyle
			\sum_{|\alpha|\leq 1}\int_{\Omega_T} M\left(z,\frac{\big|D^\alpha u(z)-D^\alpha u_n(z)\big|}{\lambda_2}\right)dz\\
			\displaystyle	\leq\frac{1}{2}\int_{\Omega_T} M\left(z,\frac{|u(z)-u_{n}(z)|}{2\lambda_1}\right)dz+\frac{1}{2}\int_{\Omega_T} M\left(z,\frac{|D u(z)-D u_{n}(z)|}{4c\lambda_1}\right)dz\\
			\displaystyle	\leq\frac{1}{2}\int_{\Omega_T}M\left(z,\frac{|u(z)-\chi_n(z)u(z)|}{2\lambda_1}\right)dz+
			\frac{1}{4}\int_{\Omega_T} \left(z,\frac{|D u(z)-\chi_n(z)D u(z)|}{2\lambda_1}\right)dz\\
			\displaystyle	+\frac{1}{4}\int_{\Omega_T} M\left(z,\frac{1}{2c\lambda_1 n}|u(z)(D\chi)\left(\frac{z}{n}\right) |\right)dz,
		\end{array}
	\end{equation*}
	all the terms on the right-hand side of the above inequality tend to zero as $n\rightarrow \infty$. Indeed, the first and the second terms tend to zero by (\ref{con_seq_ur}), while for the third one can use again Lebesgue's dominated convergence theorem, since 
	$
	M\left(z,\frac{1}{2c\lambda_1 n}\left|(D\chi)\left(\frac{z}{n}\right) u(z)\right|\right)\xrightarrow[n\to+\infty]{} 0 \mbox{  a.e. in } \Omega_T
	$
	and for $n\geq1$
	$
	\int_{\Omega_T} M\left(z,\frac{1}{2c\lambda_1 n}\left|(D\chi)(z/n) u(z)\right|\right)dz\leq
	\int_{\Omega_T} M\left(z,\frac{|u(z)|}{\lambda_1}\right)dz<\infty.
	$
	Therefore, for any $\eta>0$ and for $n$ large enough we obtain
	\begin{equation}\label{eq_cutt_off_par}
		\sum_{|\alpha|\leq 1}\int_{\Omega_T} M\left(z,\frac{\big|D^\alpha u(z)-D^\alpha u_n(z)\big|}{\lambda_2}\right) dz\leq \eta
	\end{equation}
	\par\textbf{Step 2.} In view of the first step we can assume without loss of generality that $u$ has a compact support $K$ in $\overline{\Omega}_T$. We can therefore distinguish the two cases :  $K\subset\subset \Omega_T$ and $K$ meets $\partial \Omega_T$. If $K\subset\subset \Omega_T$, then by a similar argument as \cite[Lemma 12]{AYGS2017}, there exists a function $w\in\mathcal{C}^\infty_0(\Omega_T) $ such that (\ref{eq1_th_appr_W1X0}) holds true.
	Suppose now that $K$ meets $\partial\Omega_T$. Let $\{\theta_i\}$ be the finite open covering of $\partial\Omega$ given by the segment property and $\{\vartheta_j\}$ the finite open covering of $\partial([0,T])$. Since $K\cap \partial \Omega_T$ is compact, there exist two finite collections $\{\theta_i\}_{i=1,\cdots,p}$ and $\{\vartheta_j\}_{1\leq j\leq2}$ such that $\cup_{j=1}^{2}\cup_{i=1}^{p}(\vartheta_j\times \theta_i)$ covers  $K\cap\partial \Omega_T$.
	Let
	$
	E=K\setminus \cup_{j=1}^{2}\cup_{i=1}^{p}(\vartheta_j\times \theta_i).
	$
	Since $E$ is a compact subset  of $\Omega_T$, there exist two open sets $\theta_0$ and $\vartheta_0$ with a compact closure in $\Omega$ and $(0,T)$ respectively such that $E\subset\overline{\vartheta}_{0}\times \overline{\theta}_{0}\subset \Omega_T$. Hence $\cup_{j=0}^{2}\cup_{i=0}^{p}(\vartheta_{j}\times \theta_{i})$ is an open covering of $K$.
	Moreover, we can construct another open covering  $(\vartheta^{\prime}_{j}\times \theta^{\prime}_{i})_{\substack{i=0,\cdots,p\\j=0,1,2}}$, of $K$ such that each $\theta^{\prime}_{i}$ and $\vartheta^{\prime}_{j}$ have compact closures in $\theta_i$ and $\vartheta_j$ respectively.
	Let $\{\psi_{j,i}\}$ be a partition of unity subordinate to $(\vartheta^{\prime}_{j}\times \theta^{\prime}_{i})_{\substack{i=0,\cdots,p\\ j=0,1,2}}$ with $\sum_{i=0}^p\sum_{j=0}^2\psi_{j,i}=1$ on $K$ and let $u_{j,i}=u \psi_{j,i}$. Hence
	$
	u=\sum_{i=0}^{p}\sum_{j=0}^2u_{j,i},\;
	\mbox{ supp } u_{j,i}\subset \vartheta^{\prime}_{j}\times \theta^{\prime}_{i} \mbox{ and } D^{\alpha}u_{j,i}\in L_{M}(\Omega_T),\; \forall |\alpha|\leq1.
	$
	Now by means of mollification properties we will construct the function $w$ in (\ref{eq1_th_appr_W1X0}). We distinguish the following four cases:
	\par\noindent$\bullet$ $\underline{\mbox{The case }i=j=0}$: We have $\mbox{supp}\, u_{\scriptscriptstyle0,0}\subset \vartheta^{\prime}_0\times \theta^{\prime}_0\subset \Omega_T$. So for $\varepsilon_{\scriptscriptstyle0,0}>0$ small enough such that
	$
	\varepsilon_{\scriptscriptstyle0,0}<\min\Big\{\mbox{dist}(\theta^{\prime}_0,\partial \Omega), \mbox{dist}(\vartheta^{\prime}_0,\partial (0,T))\Big\},
	$
	the regularized function defined by
	$u_{\scriptscriptstyle0,0}^{\varepsilon_{\scriptscriptstyle0,0}}=J_{\varepsilon_{\scriptscriptstyle0,0}}*u_{\scriptscriptstyle0,0}$ belongs to $\mathcal{C}^{\infty}_{0}(\Omega_T)$ and by \cite[Lemma 12]{AYGS2017}, there exist $\lambda_2>0$ and $\tilde{\varepsilon}_{\scriptscriptstyle0}$ such that for all $\varepsilon_{\scriptscriptstyle0,0}\leq\tilde{\varepsilon}_{\scriptscriptstyle0}$ we get
	\begin{equation}\label{con_seq_u0}
		\int_{\Omega_T}M\left(z,\frac{|D^\alpha u_{\scriptscriptstyle 0,0}^{\varepsilon_{\scriptscriptstyle 0,0}}(z)-D^\alpha u_{\scriptscriptstyle0,0}(z)|}{\lambda_3}\right)\leq \eta.
	\end{equation}
	Let $i,j$ be fixed. For the remaining cases, we extend $u_{j,i}$ to $\mathbb{R}^{N+1}$ to be identically zero outside $\vartheta^{\prime}_{j}\times \theta^{\prime}_{i}$. We note 
	$z_{j,i}=(z_{j},z_{i})$ the non-zero vector associated to $\vartheta_j\times \theta_i$ by the segment property and we take
	$(r_{j},r_{i})\in(0,1)^2$ be such that
	$
	0<r_{j}<\min\Big\{(|z_{j}|+1)^{-1}, \mbox{dist}(\vartheta^{\prime}_{j},\partial \vartheta_{j})|z_{j}|^{-1}\Big\}
	$
	and
	$
	0<r_{i}<\min\Big\{(|z_{i}|+1)^{-1}, \mbox{dist}(\theta^{\prime}_{i},\partial \theta_{i})|z_{i}|^{-1}\Big\}.
	$ 
	We set 
	$
	\Gamma^t_{j,r_{j}}=\overline{\vartheta}^{\prime}_{j}\cap\partial (0,T)-r_{j}z_{j}\quad\mbox{ and }\quad \Gamma^x_{i,r_{i}}=\overline{\theta}^{\prime}_{i}\cap\partial \Omega-r_{i}z_{i}.
	$
	\par\noindent$\bullet$ $\underline{\mbox{The case } 1\leq i\leq p\mbox{ and }j=1,2}$ : Let
	$$
	\varepsilon_{j,i}<\min\Big\{\mbox{ dist}\left(\Gamma^t_{j,r_{j}}, \vartheta_j\cap[0,T]\right), \mbox{dist}\left(\Gamma^x_{i,r_{i}}, \theta_i\cap\overline{\Omega}\right)\Big\}	
	$$
	and set $
	(u_{j,i})_{(r_{j},r_{i})}(t,x)=u_{j,i}\left(t+r_{j}z_{j},x+r_{i}z_{i}\right).$	
	Let $B(0,1)$ be the open ball unit in $\mathbb{R}^{N+1}$, for $i=1,\cdots,p$ and $j=1,2$ we define the sequence of mollification
	$$
	\begin{array}{lll}	\displaystyle
		\left(u_{j,i}^{\varepsilon_{j,i}}\right)_{(r_{j},r_{i})}(z)&=\displaystyle J_{\varepsilon_{j,i}}*(u_{j,i})_{(r_{j},r_{i})}(z)\\
		&=\displaystyle\int_{B(0,1)} J(\tau,y)u_{j,i}\left(t+r_{j}z_{j}-\varepsilon_{j,i}\tau,x+r_{i}z_{i}-\varepsilon_{j,i} y\right)d\tau dy.
	\end{array}	
	$$
	Therefore, by \cite[Theorme 2]{AYGS2017} the sequence $(u_{j,i}^{\varepsilon_{j,i}})_{(r_{j},r_{i})}$ belongs to $C^\infty_0(\overline{\Omega}_T)$ and there exist $\lambda_4>0$ and $\tilde{\varepsilon}_{0}$ and $\tilde{r}_{0}$ such that for all $\varepsilon_{j,i}\leq\tilde{\varepsilon}_{0}$ and all
	$r_{j},r_{i}\leq\tilde{r}_{0}$ we have
	\begin{equation}\label{con_seq_u1}
		\int_{\Omega_T}M\left(z,\frac{|D^\alpha \left(u_{j,i}^{\varepsilon_{j,i}}\right)_{(r_{j},r_{i})}(z)-D^\alpha u_{j,i}(z)|}{\lambda_4}\right)dz\leq \eta.
	\end{equation}
	\par\noindent$\bullet$ $\underline{\mbox{The case } i=0\mbox{ and }j=1,2}$ : Let
	$
	\varepsilon_{j,0}<\min\Big\{\mbox{ dist}\left(\Gamma^t_{j,r_{j}}, \vartheta_j\cap[0,T]\right),\mbox{ dist}(\theta^{\prime}_0,\partial \Omega)\Big\}.
	$
	By virtue of \cite[Theorme 2]{AYGS2017}, the sequence
	$$
	\begin{array}{llll}
		\left(u_{j,0}^{\varepsilon_{j,0}}\right)_{(r_{j},0)}=J_{\varepsilon_{j,0}}*(u_{j,0})_{(r_{j},0)}
		=\displaystyle\int_{B(0,1)} J(\tau,y)u_{j,0}\left(t+r_{j}z_{j}-\varepsilon_{j,0}\tau,x-\varepsilon_{j,0} y\right)d\tau dy
	\end{array}
	$$
	belongs to $C^\infty_0([0,T]\times\Omega)$ and there exist $\lambda_5>0$ and $\tilde{\varepsilon}_{0}$ and $\tilde{r}_{0}$ such that for all $\varepsilon_{j,0}\leq\tilde{\varepsilon}_{0}$ and for all
	$r_{j}\leq\tilde{r}_{0}$, we have
	\begin{equation}\label{con_seq_u2}
		\int_{\Omega_T}M\left(z,\frac{|D^\alpha \left(u_{j,0}^{\varepsilon_{j,0}}\right)_{(r_{j},0)}(z)-D^\alpha u_{j,0}(z)|}{\lambda_5}\right)dz\leq \eta.
	\end{equation}
	\par\noindent$\bullet$ $\underline{\mbox{The case }1\leq i\leq p \mbox{ and } j=0}$ : Let
	$
	\varepsilon_{0,i}<\min\Big\{ \mbox{dist}(\vartheta^{\prime}_0,\partial (0,T)),\mbox{dist}\left(\Gamma^x_{i,r_{i}}, \theta_i\cap\overline{\Omega}\right)\Big\}.
	$
	Again by \cite[Theorme 2]{AYGS2017}, the sequence
	$$
	\left(u_{0,i}^{\varepsilon_{0,i}}\right)_{(0,r_{i})}=J_{\varepsilon_{0,i}}*(u_{0,i})_{(0,r_{i})}=\int_{B(0,1)} J(\tau,y)u_{0,i}(t-\varepsilon_{0,i}\tau,x+r_{i}z_{i}-\varepsilon_{0,i} y)d\tau dy,
	$$
	belongs to $C^\infty_0((0,T)\times\overline{\Omega})$ and there exist $\lambda_6>0$ and $\tilde{\varepsilon}_{0}$ and $\tilde{r}_{0}$ such that for all $\varepsilon_{0,i}\leq\tilde{\varepsilon}_{0}$ and all
	$r_{i}\leq\tilde{r}_{0}$ we have
	\begin{equation}\label{con_seq_u3}
		\int_{\Omega_T}M\left(z,\frac{|D^\alpha \left(u_{0,i}^{\varepsilon_{0,i}}\right)_{(0,r_{i})}(z)-D^\alpha u_{0,i}(z)|}{\lambda_6}\right)dz\leq \eta.
	\end{equation}
	Let us now define the function
	\begin{equation}\label{seq_sp1}
		w=\sum_{i=0}^p\sum_{j=0}^2 (u_{j,i}^{\varepsilon_{j,i}})_{(r_{j},r_{i})},
	\end{equation}
	where we use the convention  $r_0=0$ and  set $\left(u_{0,0}^{\varepsilon_{0,0}}\right)_{(r_{0},r_{0})}:=u_{0,0}^{\varepsilon_{0,0}}$. Note that the function $w$ belongs to $\mathcal{C}^\infty_0(\overline{\Omega}_T)$. Let $\lambda_7=\max\{\lambda_k;\; k=3,\cdots,6\}$, the convexity of  $M(z,\cdot)$ yields
	$$
	\begin{array}{c}\displaystyle
		\int_{\Omega_T}M\left(z,\frac{|D^{\alpha} u- D^{\alpha}w|}{2^{(p+4)}\lambda_7}\right)dz
		\leq\sum_{i=0}^p\sum_{j=0}^2\int_{\Omega_T}M\left(z,\frac{|D^{\alpha} u_{j,i}- D^{\alpha}(u_{j,i}^{\varepsilon_{j,i}})_{(r_{j},r_{i})}|}{\lambda_7}\right)dz
	\end{array}
	$$
	and also
	$
	\|u-w\|_{L^2(\Omega_T)}\leq\sum_{i=0}^p\sum_{j=0}^2
	\big\|u_{j,i}-\left(u_{j,i}^{\varepsilon_{j,i}}\right)_{(r_{j},r_{i})}\big\|_{L^2\Omega_T}.
	$
	Therefore, combining (\ref{con_seq_u0}), (\ref{con_seq_u1}), (\ref{con_seq_u2}), (\ref{con_seq_u3}), we get
	\begin{equation}\label{con_seq_u4}
		\int_{\Omega_T}M\left(z,\frac{|D^{\alpha} u(z)-D^{\alpha} w(z)|}{2^{(p+4)}\lambda_7}\right)dz\leq\eta.
	\end{equation}
	Finally, for $\lambda=\max\{2\lambda_2,2^{(p+5)}\lambda_7\}$, by (\ref{eq_cutt_off_par}), (\ref{con_seq_u4}) and the convexity of $M$ we obtain (\ref{eq1_th_appr_W1X0}).
	\par Now if $u$ belongs to $W^{1,x}E_{M}(\Omega_T)$, by similar arguments as above with arbitrary $\lambda,\lambda_k$, $k=1,\cdots,7$, the above modular convergences become norm ones in $L_M(\Omega_T)$. Then we conclude thanks to \cite[Lemma 2.7]{AFY_ms19}.
\end{proof}
\begin{theorem}\label{lm_appr_spt2}
	Let $\Omega$ be an open subset of $\mathbb{R}^N$  and let $(M,M^\ast)$ be a pair of complementary $\Phi(\Omega_T)$-functions  satisfying booth the condition (\ref{incBM}) on $\Omega_T$. Then 
	$
		W^{1,x}_{0}L_{M}(\Omega_T)=\overline{\mathcal{C}^{\infty}_0(\Omega_T)}^{\sigma\left(\Pi L_M, \Pi E_{M^\ast}\right)}.
	$
	That is $W^{1,x}_{0}L_{M}(\Omega_T)$ is the closure of $\mathcal{C}^{\infty}_0(\Omega_T)$ in $W^{1,x}L_{M}(\Omega_T)$ with respect to the weak-$^\ast$ topology  $\sigma(\Pi L_M, \Pi E_{M^\ast})$.
\end{theorem}
\begin{proof} 
	We only prove the inclusion $W^{1,x}_{0}L_{M}(\Omega_T)\subset \overline{\mathcal{C}^{\infty}_0(\Omega_T)}^{\sigma\left(\Pi L_M, \Pi E_{M^\ast}\right)}$, the inverse inclusion being trivial. Pick a function $u\in W^{1,x}_{0}L_{M}(\Omega_T)$  and let $\eta>0$ be arbitrary. We shall prove that  there exists a function $\overline{u}\in \mathcal{C}^\infty_0(\Omega_T)$  such that
	$$
	\left|\int_{\Omega_T}(u(z)-\overline{u}(z))w_0(z)dz\right|\leq\eta\quad
	\mbox{	and }\quad
	\left|\int_{\Omega_T}(\nabla u(z)-\nabla \overline{u}(z)):w(z)dz\right|\leq\eta,
	$$
	for all $(w_0,w)\in E_{M^\ast}(\Omega_T)\times E_{M^\ast}(\Omega_T)^{N}$. Let first assume that $w_0$ and $w$ be arbitrary respectively in $ \mathcal{C}^{\infty}_0(\Omega_T)$ and $ \mathcal{C}^{\infty}_0(\Omega_T)^{N}$.
	Since for $t\in(0,T)$ fixed the function $u(t,\cdot)\in W^{1}_{0}L_{M}(\Omega)$,  there is a sequence $v(t,\cdot)\in \mathcal{C}^\infty_0(\Omega)$ such that 
	$$
	\left|\int_{\Omega}(u(t,x)-v(t,x))w_0(t,x)dx\right|\leq\frac{\eta}{T}
	\mbox{	and }
	\left|\int_{\Omega}(\nabla u(t,x)-\nabla v(t,x)):w(t,x)dx\right|\leq\frac{\eta}{T}.
	$$
	Recall here that  since 
	$ \mathcal{C}^{\infty}_0(\Omega_T)\subset \mathcal{C}^{\infty}_0([0,T],\mathcal{C}^{\infty}_0(\Omega))$
	and since  
	$
	\mathcal{C}^{\infty}_0(\Omega_T)^{N}\subset \mathcal{C}^{\infty}_0([0,T],\mathcal{C}^{\infty}_0(\Omega)^{N})
	$
	the functions $w_0(t,\cdot)$ and $w(t,\cdot)$ are arbitrary in $\mathcal{C}^{\infty}_0(\Omega)$ and $\mathcal{C}^{\infty}_0(\Omega)^{N}$ respectively.  Then integrating over $(0,T)$ we get
	$$
		J_1:=\left|\int_{\Omega_T}|(u(z)-v(z))w_0(z)dz\right|\leq\eta
	\mbox{	and }
	I_1:=\left|\int_{\Omega_T}(\nabla u(z)-\nabla v(z)):w(z)dz\right|\leq\eta.
	$$
	For $x\in\Omega$ fixed, we define the sequence
	$v_{j}(\cdot,x)= v(\cdot,x)\chi_{K_j}(\cdot),$
	where $\chi_{K_j}$ stands for the characteristic function with respect to the time variable $t$ of the set
	$$
	K_j=\left\{t\in I,\; |t|\leq j,\; \mbox{dist}\left(t,]-\infty,0]\cup[T,+\infty[\right)\geq \frac{1}{j}\right\}.
	$$
	It is clear that $\{v_j\}_j$ is a sequence of functions with compact support in $\Omega_T$. Moreover observe that
	$
	| (v_{j}-v)w_0|\leq 2|v||w_0|\in L^1(\Omega_{T})
	$
	and	$v_{j}\rightarrow v$ a.e. in  $\Omega_T$ so that by Lebesgue's dominated convergence theorem there is $j_0>0$ large enough such that
	$$
	J_2:=\left|\int_{\Omega_T} (v(z)- v_{j_0}(z))w_0(z)dz\right|\leq\eta.
	$$
	In a similar way we also have 
	$
	I_2:=\left|\int_{\Omega_T}(\nabla v(z)-\nabla v_{j_0}(z)):w(z)dz\right|\leq\eta.
	$
	Let $J$ stands for the Friedrichs mollifiers kernel defined on $\mathbb{R}$ such that $\int_{\mathbb{R}}J(t)dt=1$. 
	For $\varepsilon>0,$ we  define 
	$$
	v_{\varepsilon}(t,x)=J_\varepsilon\ast 	v_{j_0}(t,x)=\int_{\mathbb{R}}J_\varepsilon(t-\tau)	v_{j_0}(\tau,x)d\tau
	=\int_{B(0,1)}	v_{j_0}(t-\varepsilon\tau,x)J(\tau)d\tau.
	$$
	For $\varepsilon<\mbox{dist }(\mbox{supp }v_{j_0},\partial [0,T])$ the sequence $v_{\varepsilon}$ belongs to $ \mathcal{C}^\infty_0(\Omega_T)$. Using Fubini's theorem we can write
	$$
	\begin{array}{lll}\displaystyle
		J_3:=\displaystyle\left|\int_{\Omega_T}(v_{\varepsilon}(z)-v_{j_0}(z))w_0(z)dz\right|\leq\displaystyle\int_{\Omega_T}\left|(v_{\varepsilon}(z)-v_{j_0}(z))w_0(z)\right|dz\\
		\leq\displaystyle\int_{\Omega_T}\int_{B(0,1)}J(\tau)|(v_{j_0}(t-\varepsilon\tau,x)-v_{j_0}(t,x))w_0(z)d\tau dz\\
		\leq\displaystyle\int_{\Omega_T}\int_{B(0,1)}J(\tau)\left|(v_{j_0}(t-\varepsilon\tau,x)-v_{j_0}(t,x))w_0(z)\right|d\tau dz\\
		\leq\displaystyle\int_{B(0,1)}J(\tau)\int_{\Omega_T}|(v_{j_0}(t-\varepsilon\tau,x)-v_{j_0}(t,x))w_0(z)|dzd\tau\\
		\leq\displaystyle\|w_0\|_{L^{\infty}(\Omega_{T})}\int_{B(0,1)}J(\tau)\int_{\Omega_T}|(v_{j_0}(t-\varepsilon\tau,x)-v_{j_0}(t,x))|dzd\tau.
	\end{array}
	$$
	Thanks to \cite[Theorem 2.4.2]{book_KJF} there is $\rho>0$ such that whenever $\varepsilon\tau<\rho$ it holds
	$
	\int_{\Omega_T}|(v_{j_0}(t-\varepsilon\tau,x)-v_{j_0}(t,x))|dz\leq \frac{\eta}{\|w_0\|_{L^{\infty}(\Omega_{T})}}.
	$
	Hence, the choice $\varepsilon<\min\left(\mbox{dist }(\mbox{supp }v_{j_0},\partial [0,T]),\frac{\rho}{T}\right)$ yields
	$$
	J_3:=\left|\int_{\Omega_T}(v_{\varepsilon}(z)-v_{j_0}(z))w_0(z)dz\right|\leq \eta.
	$$
	Similarly, since $\nabla v_{\varepsilon}(t,x)=J_\varepsilon\ast\nabla v_{j_0}(t,x)$ for  $\varepsilon$ sufficiently small  we obtain 
	$$
	I_3:=\left|\int_{\Omega_T}(\nabla v_{\varepsilon}(z)-\nabla v_{j_0}(z)):w(z)dz\right|\leq \eta.
	$$
	Therefore, combining $J_1$, $J_2$ and $J_3$ and then $I_1$, $I_2$ and  $I_3$ for  $\varepsilon$ sufficiently small  we get
	\begin{equation}\label{eq_est}
		\left|	\int_{\Omega_T}(u(z)-v_{\varepsilon}(z))w_0(z)dz\right|\leq3\eta \mbox{ and }
		\left|	\int_{\Omega_T}(\nabla u(z)-\nabla v_{\varepsilon}(z)):w(z)dz\right|\leq3\eta.
	\end{equation}
	Finally, since $M^{\ast}$ satisfies (\ref{incBM}) on $\Omega_T$ by \cite[Theorem 2.2]{AY2018} smooth functions compactly supported in $\Omega_{T}$ are norm dense in $E_{M^{\ast}}(\Omega_T)$ the estimations (\ref{eq_est})  hold true for all $w_0\in E_{M^\ast}(\Omega_T)$ and $w\in  E_{M^\ast}(\Omega_T)^{N}$. So we obtain the following inclusion 
	$$
	W^{1,x}_{0}L_{M}(\Omega_T)\subset\overline{\mathcal{C}^{\infty}_0(\Omega_T)}^{\sigma\left(\Pi L_M ,\Pi E_{M^\ast}\right)}.
	$$
	This achieves the proof.
\end{proof}
\begin{theorem}\label{th_appr_spt2}
	Let $\Omega$ be an open subset of $\mathbb{R}^N$ satisfying the segment property and let $(M,M^\ast)$ be a pair of complementary $\Phi(\Omega_T)$-functions such that $M$ satisfy (\ref{X1}) on $\Omega_T$ and $M^\ast$ satisfies (\ref{incBM}) on $\Omega_T$. Then for any $u\in 	W^{1,x}_{0}L_{M}(\Omega_T)$, there exists a sequence of functions $u_k\in \mathcal{C}^\infty_0(\Omega_T)$ such that $D^\alpha u_k\xrightarrow[k\to\infty]{} D^\alpha u$, $|\alpha|\leq 1$, in modular sense in $L_M(\Omega_T)$.
\end{theorem}
\begin{proof} 
	In view of Theorem \ref{lm_appr_spt2}, and since $\Omega_T$ satisfies the segment property the proof flows exactly by a similar argument as \cite[Theorem 3]{AYGS2017}.
\end{proof}
\begin{corollary}\label{cor_appr_spt} 
	Let $\Omega$ be an open subset of $\mathbb{R}^N$, $N\geq1$, satisfying the segment property and let $(M,M^\ast)$ be a pair of complementary $\Phi(\Omega_T)$-functions such that $M$ satisfies (\ref{X1}) on $\Omega_T$ and $M^\ast$ satisfies (\ref{incBM}) on $\Omega_T$. For every $u\in 	W^{1,x}_{0}L_{M}(\Omega_T)$ 
	(resp. $u\in W^{1,x}_{0}E_{M}(\Omega_T)$) there exists a sequence $\{u_k\}_k\subset\mathcal{C}^\infty_0([0,T];\mathcal{C}^\infty_0(\Omega))$ 
	such that $D^\alpha u_k\xrightarrow[k\to\infty]{}D^{\alpha} u$, $|\alpha|\leq1$, modularly (resp. in norm) in $L_M(\Omega_T)$.
\end{corollary}	
\begin{proof} 
Observing that $\mathcal{C}^\infty_0(\Omega_T)\subset\mathcal{C}^{\infty}_{0}([0,T];\mathcal{C}^{\infty}_{0}(\Omega))$ we can easily check that the set $\mathcal{C}^{\infty}_{0}([0,T];\mathcal{C}^{\infty}_{0}(\Omega))$ is dense modularly in $W^{1,x}_{0}L_{M}(\Omega_T)$ thanks to Theorem \ref{th_appr_spt2}. Nonetheless, for later use, we give an explicit construction of a sequence of $\mathcal{C}^{\infty}_{0}([0,T];\mathcal{C}^{\infty}_{ 0}( \Omega))$ that converges to an element of $W^{1,x}_{0}L_{M}(\Omega_T)$.\\
	Indeed following exactly the same lines as in Theorem \ref{th_appr_spt}, except that for $i=1,\cdots,p$ we take
	$$
	\varepsilon_{j,i}<\min\Big\{
	\mbox{dist}\left(\Gamma^t_{j,r_{j}}, \vartheta_i\cap[0,T]\right), \mbox{ dist}\left((\theta^{\prime}_i\cap \overline{\Omega})+r_{i}z_{i}, \mathbb{R}^N\setminus\Omega\right)\Big\}
	$$
	and we substitute $r_{i}$ by $-r_{i}$ in (\ref{seq_sp1}), that is
	\begin{equation}\label{seq_spt2}
		\begin{array}{lll}\displaystyle
			w&=\displaystyle\sum_{i=0}^p\sum_{j=0}^2(u_{j,i}^{\varepsilon_{j,i}})_{ (r_{j},-r_{i})}=\sum_{i=0}^p\sum_{j=0}^2J_{\varepsilon_{j,i}}*(u_{j,i})_{(r_{j},-r_{i})}\\
			&=\displaystyle\sum_{i=0}^p\sum_{j=0}^2\int_{B(0,1)} J(\tau,y)u_{j,i}(t+r_{j}z_{j}-\varepsilon_{j,i}\tau,x-r_{i}z_{i}-\varepsilon_{j,i} y)d\tau dy
		\end{array}
	\end{equation}
	The function $w$ belongs to $\mathcal{C}^{\infty}_{0}([0,T];\mathcal{C}^{\infty}_{0}(\Omega))$  and satisfies (\ref{eq1_th_appr_W1X0}).
\end{proof}
\begin{theorem}\label{cor_appr_spt2}
	Let $\Omega$ be an open subset of $\mathbb{R}^N$ and let $(M,M^\ast)$ be a pair of complementary $\Phi(\Omega_T)$-functions such that $M$ satisfies (\ref{X1}) on $\Omega_T$ and $M^\ast$ satisfies (\ref{incBM}) on $\Omega_T$. Then for any
	$u\in W^{1,x}_{0}E_{M}(\Omega_T)$, there exists a sequence of functions $u_k\in \mathcal{C}^\infty_0(\Omega_T)$ such that $D^\alpha u_k\xrightarrow[k\to\infty]{} D^\alpha u$, $|\alpha|\leq 1$, in norm in $L_M(\Omega_T)$.
\end{theorem}
\begin{proof} 
Set $z:=(t,x)$ and let $\lambda>0$ be arbitrary. In view of \cite[Lemma 2.7]{AFY_ms19}  it is sufficient to show that  for an arbitrary $\eta>0$ there exists a sequence $w\in \mathcal{C}^\infty_0(\Omega_T)$ such that
	$$
	\int_{\Omega_T}M\left(z,\frac{|D^{\alpha}u(z)-D^{\alpha}w(z)|}{\lambda}\right)dz\leq\eta.
	$$
	Let $u\in W^{1,x}_{0}E_{M}(\Omega_T)$. One has $\int_{\Omega_T} M\left(z,\frac{|D^{\alpha}u(z)|}{\lambda}\right)dz<\infty$. As for a fixed $t\in(0,T)$ we have $u(t,\cdot)\in W^1_0E_{M}(\Omega)$,  for an arbitrary $\eta>0$ there is a function $v(t,\cdot)\in \mathcal{C}^\infty_0(\Omega)$ such that  for all $|\alpha|\leq1$ we have 
	$
	\frac{1}{\lambda}\|D^{\alpha} u(t)-D^{\alpha} v(t)\|_{L_M(\Omega)}\leq\frac{\eta}{3T}.
	$
	For $0<\eta<3T$ we have
	$
	\int_{\Omega}M\left(t,x,\frac{|D^{\alpha} u(t,x)-D^{\alpha} v(t,x)|}{\lambda}\right)dx\leq\frac{\eta}{3T}.
	$
	Integrating over $(0,T)$ we get
	$$
	J_1:=\int_{\Omega_T}M\left(z,\frac{|D^{\alpha} u(z)-D^{\alpha} v(z)|}{\lambda}\right)dz\leq\frac{\eta}{3}.
	$$
	For $x\in\Omega$ fixed, we define the sequence
	$
	v_{j}(\cdot,x)= v(\cdot,x)\chi_{K_j}(\cdot),
	$
	where $\chi_{K_j}$ stands for the characteristic function with respect to the time variable $t$ of the set
$$
K_j=\left\{t\in I,\; |t|\leq j,\; \mbox{dist}\left(t,]-\infty,0]\cup[T,+\infty[\right)\geq \frac{1}{j}\right\}.
$$
	On  one hand we have
	$$
	\begin{array}{lll}
		M\left(z,\frac{|D^{\alpha} v_{j}(z)-D^{\alpha} v(z)|}{4\lambda}\right)
		&\leq M\left(z,\frac{|D^{\alpha}v(z)|}{2\lambda}\right)\\
		&\leq \left(M\left(z,\frac{|D^{\alpha}v(z)-D^{\alpha}u(z)|}{\lambda}\right)
		+ M\left(z,\frac{|D^{\alpha}u(z)|}{\lambda}\right)\right)\in L^1(\Omega_T).
	\end{array}
	$$
	On the other hand, since the sequence $\{D^{\alpha} v_{j}\}_j$ converges to $D^{\alpha} v$ almost everywhere in $\Omega_T$ as $j$ tends to infinity, we have
	$
	M\left(z,\frac{|D^{\alpha} v_{j}(z)-D^{\alpha} v(z)|}{4\lambda}\right)\rightarrow 0, \mbox{ a.e. in } \Omega_T.
	$
	Hence, by Lebesgue's dominated convergence theorem there is $j_0>0$ large enough such that
	$$
	J_2:=\int_{\Omega_T}M\left(z,\frac{|D^{\alpha} v(z)-D^{\alpha} v_{j_0}(z)|}{4\lambda}\right)dz\leq\frac{\eta}{3}.
	$$
	Note that  the function $v_{j_0}\in  W^1E_{M}(\Omega_T)$ and has a compact support in $\Omega_T$, so that by \cite[Lemma 12]{AYGS2017} for $\eta>0$, with $\eta/3<1$, there exist a sequence  $\{w_{\varepsilon}\}\in \mathcal{C}^\infty_0(\Omega_T)$ and  $\varepsilon_0>0$ such that for every $\varepsilon\leq\varepsilon_0$ we have
	$$
	J_3:=\int_{\Omega_T}M\left(z,\frac{|D^{\alpha}v_{j_0}(z)-D^{\alpha}
		w_{\varepsilon}(z)|}{\lambda}\right)dz\leq \frac{1}{\lambda}\|D^{\alpha}v_{j_0}-D^{\alpha}
	w_{\varepsilon}\|_{L_M(\Omega_{T})}\leq\frac{\eta}{3}.
	$$
	Then,  combining $J_1$, $J_2$ and $J_3$ and using the convexity of $M(z,\cdot)$, for  $\varepsilon\leq\varepsilon_0$  we get $\int_{\Omega_T}M\left(z,\frac{|D^{\alpha}u(z)-D^{\alpha}
		w_{\varepsilon}(z)|}{\lambda}\right)dz\leq\eta$	for arbitrary $\lambda>0$.	This achieves the proof.
\end{proof}
\begin{theorem}\label{proposition_equality}
	Let $\Omega$ be a bounded open subset in $\mathbb{R}^N$ having the segment property and let $M\in\Phi(\Omega_T)$. Assume that $M$ satisfies (\ref{X1}) on $\Omega$ and it's complementary $M^\ast$ satisfies (\ref{incBM}) on $\Omega$.
	Then we get 
	$$
	W^{1,x}_{0}L_{M}(\Omega_T)=\Big\{u:(0,T)\to W^{1,1}_0(\Omega): u\in L_M(\Omega_T), \nabla u\in L_M(\Omega_T)^N\Big\}.
	$$
\end{theorem}
\begin{proof} of Theorem \ref{proposition_equality}.
	For $u\in W^{1,x}_{0}L_{M}(\Omega_T)$ one has   $u(t)\in W^1_0L_{M}(\Omega)$ for a.e $t\in (0,T)$. Then, thanks to \cite[Theorem 1.4]{AY2018} it follows that
	$
	W^{1}_{0}L_M(\Omega)=W^{1,1}_{0}(\Omega)\cap W^{1}L_M(\Omega).
	$
	Hence we get the following inclusion 
	$
	W^{1,x}_{0}L_{M}(\Omega_T)\subset\Big\{u:(0,T)\to W^{1,1}_0(\Omega): u\in L_M(\Omega_T), \nabla u\in L_M(\Omega_T)^N\Big\}.
	$ 
	On the other hand, since for all $|\alpha|\leq1$ we have $|D^\alpha u|\in L_M(\Omega_T)$, by Fubini's theorem we have $u(t)\in W^1L_M(\Omega)$ for a.e $t\in (0,T)$. Which proves the reverse inclusion.
\end{proof}
\section{Approximation in a dual space}\label{Sec_app time}
The main result of this section is Theorem \ref{th_appr_W_spt} where a density result in a dual space is given. We begin first by the following duality result.
\begin{theorem}\label{th_dual_MBS}
	Let $(M,M^\ast)$ be a pair of complementary $\Phi(\Omega_T)$-functions such that $M$ satisfies (\ref{X1}) on $\Omega_T$ and $M^\ast$ satisfies (\ref{incBM}) on $\Omega_T$.
	Then, the dual space of $W^{1,x}_{0}E_{M}(\Omega_T)$,
	denoted $(W^{1,x}_{0}E_{M}(\Omega_T))^\prime$, is isometrically isomorphic to the space $W^{-1,x}L_{M^\ast}(\Omega_T)$ defined by
$$
\begin{array}{lll}
W^{-1,x}L_{M^\ast}(\Omega_T):=&
\Big\{u:(0,T)\to W^{-1}L_{M^\ast}(\Omega): u=u_0-\operatorname{div } U,\\
&\mbox{ with } u_0\in L_{M^\ast}(\Omega_T)\mbox{ and } U=(u_{i})_{\substack{1\leq i\leq N}}\in L_{M^\ast}(\Omega_T)^N \Big\}
\end{array}
$$
endowed with the norm
\begin{equation}\label{norm_lm2+}
	\|u\|_{W^{-1,x}L_{M^\ast}(\Omega_T)}=
\inf\Big\{\|u_0\|_{L_{M^{\ast}}(\Omega_T)}+\sum_{i=1}^N\|u_{i}\|_{L_{M^{\ast}}(\Omega_T)},\;/ u=u_0-\operatorname{div }U\Big\}.
\end{equation}
\end{theorem}
\begin{proof}
	In view of the density of $\mathcal{C}^\infty_0(\Omega_T)$-functions in $W^{1,x}_{0}E_{M}(\Omega_T)$ (see Theorem \ref{cor_appr_spt2}), the proof follows directly by an argument similar to \cite[Theorem A.1.]{AY_var19}.
\end{proof}
\begin{definition}\label{def_evolspac}
	Let $\Omega$ be a bounded open subset of $\mathbb{R}^N$, $N\geq1$, satisfying the segment property. Let $(M,M^\ast)$ be a pair of complementary $\Phi(\Omega_T)$-functions such that $M$ satisfies (\ref{X1}) on $\Omega_T$ and $M^\ast$ satisfies (\ref{incBM}) on $\Omega_T$. We define the space
	$$
	\begin{array}{c}
		\mathbf{W}\left(\Omega_T\right):=
		\Big\{u\in W^{1,x}_{0}L_{M}(\Omega_T)\cap L^2(\Omega_T),\; \partial_tu\in  W^{-1,x}L_{M^{\ast}}^{2}(\Omega_T)+L^2(\Omega_T)\Big\}.
	\end{array}
	$$
\end{definition}
In view of the definition of the norm on $W^{-1,x}L_{M^{\ast}}(\Omega_T)+L^2(\Omega_T)$ (see  (\ref{norm_lm2+})) we can define the following modular convergence.
\begin{definition}
	Let $\Omega$ be an open subset in $\mathbb{R}^N$, $N\geq1$, having the segment property. Let $(M,M^\ast)$ be a pair of complementary $\Phi(\Omega_T)$-functions. Assume that $M$ satisfies (\ref{X1}) on $\Omega_T$ and $M^\ast$ satisfies (\ref{incBM}) on $\Omega_T$. Let $u, u_k\in W^{-1,x}L_{M^{\ast}}(\Omega_T)+L^2(\Omega_T)$ be such that
	$$
	u=\sum_{l=1}^N\partial_{x_l} u^{a} +u^{b}+v\mbox{ and } u_k=\sum_{l=1}^N\partial_{x_l}u^{a}_{k} +u^{b}_k+v_k,
	$$
	where $u^{a},u^{a}_{k},u^{b},u^{b}_{k}  \in  L_{M^{\ast}}(\Omega_T)$ and $v,v_k\in L^2(\Omega_T)$.
	We said that the sequence $u_k$ converges modularly to $u$ in $W^{-1,x}L_{M^{\ast}}(\Omega_T)+L^2(\Omega_T)$ if $u^{a}_k$ and $u^{b}_k$ converge respectively to $u^{a}$ and $u^{b}$ as $k\to\infty$, in modular sense in $L_{M^\ast}(\Omega_T)$ and  $v_k\xrightarrow[k\to \infty]{} v$ in norm in $L^2(\Omega_T)$.
\end{definition}
\begin{theorem}\label{th_appr_W_spt}
	Let $\Omega$ be an open subset of $\mathbb{R}^N$, $N\geq1$, satisfy the segment property. Let $(M,M^\ast)$ be a pair of complementary $\Phi(\Omega_T)$-functions satisfying both (\ref{X1}) on $\Omega_T$. Then for every $u\in   \mathbf{W}\left(\Omega_T\right)$ there exists a sequence $u_k\in \mathcal{C}^\infty_0([0,T]; \mathcal{C}^\infty_0(\Omega))$ such that
	$$
	\begin{array}{lll}\displaystyle
		\bullet\; D^\alpha u_k\xrightarrow[k\to\infty]{} D^\alpha u, |\alpha|\leq 1, &\mbox{ modularly in } L_{M}(\Omega_T),\\
		\bullet\; \partial_{t}u_k \xrightarrow[k\to\infty]{} \partial_{t} u &\mbox{ modularly in } W^{-1,x}L_{M^{\ast}}(\Omega_T)+L^2(\Omega_T),\\
		\bullet\; u_k\xrightarrow[k\to\infty]{} u&\mbox{ in norm in } L^2(\Omega_T).
	\end{array}
	$$
\end{theorem}
\begin{proof}
	Set $z:=(t,x)$ and let $u\in \mathbf{W}\left(\Omega_T\right)$.  Let us write
	$
	\partial_{t}u:=\frac{\partial u}{\partial t}=\sum_{l=1}^N\partial_{x_l}u^{a} +u^{b}+v
	$
	where $u^{a},u^{b}\in L_{M^{\ast}}(\Omega_T)$ and $v\in L^2(\Omega_T)$. In particular, the function $u$ belongs to $W^{1,x}_{0}L_{M}(\Omega_T)\cap L^2(\Omega_T)$. So by Corollary \ref{cor_appr_spt} there exists a constant number $\nu>0$ such that for every $\eta>0$ there is a function $w\in \mathcal{C}^\infty_0([0,T]; \mathcal{C}^\infty_0(\Omega))$, given by (\ref{seq_spt2}), such that one has
	\begin{equation}\label{est_1}
		\begin{array}{lll}\displaystyle
			\int_{\Omega_T}M\left(z,\frac{|D^\alpha w(z)-D^\alpha u(z)|}{\nu}\right)dz\leq\eta, \mbox{ for every }|\alpha|\leq 1
			\mbox{ and }
			\|w-u\|_{L^2(\Omega_T)}\leq\eta.
		\end{array}
	\end{equation}
	Therefore, it remains to check that $\partial_tw$ converges to $\partial_tu$ modularly in the dual space $W^{-1,x}L_{M^{\ast}}(\Omega_T)+L^2(\Omega_T)$. The proof is divided into two steps. In the first step we approach $\partial_{t}u$ by time derivative of a sequence of functions $\{u_n\}_n$ having compact support in $\overline{\Omega}_T$ and then in the next step we approach $\partial_{t}w$ by $\partial_{t}u_n$  using some previous approximation results.
	\par\noindent\textbf{Step1. }  Note that the sequence $\{u_n\}_n$ defined in (\ref{fnc:uR}) belongs to $ \mathbf{W}\left(\Omega_T\right)$. We will check that $\partial_t u_n$ converges to $\partial_t u$ modularly in $W^{-1,x}L_{M^{\ast}}(\Omega_T)+L^2(\Omega_T)$. We can write
	\begin{equation}\label{eq_est_un}
		\begin{array}{lll}\displaystyle
			\partial_t u_n(z)&=&\displaystyle u(z)\partial_t\chi_n(z)+\chi_n(z)\partial_tu(z)\\
			&=&\displaystyle \frac{u(z)}{n}(\partial_t\chi)\left(\frac{z}{n}\right)+\chi_n(z)\left(\sum_{l=1}^N\partial_{x_l}
			u^{a}(z)+u^{b}(z)+v\right)\\
			&=&\displaystyle\frac{u(z)}{n}(\partial_t\chi)\left(\frac{z}{n}\right)+v\chi_n(z)+\sum_{l=1}^N\partial_{x_l}\left(\chi_n(z)u^{a}(z)\right)+u^{b}(z)\chi_n(z)\\
			&-&\displaystyle\sum_{l=1}^N \frac{u^{a}(z)}{n}(\partial_{x_l}\chi)\left(\frac{z}{n}\right).
		\end{array}
	\end{equation}
	Since $u^{a}\in L_{M^\ast}(\Omega_T)$ there exists a constant number $\delta_1>0$ such that $\displaystyle\int_{\Omega_T}M^\ast(z,|u^{a}(z)|/\delta_1)dz<\infty$. On the one hand, one has $\displaystyle M^\ast\left(z,\frac{|u^{a}(z)-\chi_n(z)u^{a}(z)|}{2\delta_1}\right)\xrightarrow[n\to\infty]{} 0\quad \mbox{  a.e. in } \Omega_T$, meanwhile on the other hand $\displaystyle M^\ast\left(z,\frac{|u^{a}(z)-\chi_n(z)u^{a}(z)|}{2\delta_1}\right) \leq M^\ast\left(z,\frac{|u^{a}(z)|}{\delta_1}\right)\in L^1(\Omega_T)$. Hence, by the Lebesgue dominated convergence theorem it follows $\displaystyle \lim_{n\rightarrow\infty}\int_{\Omega_T}M^\ast\left(z,\frac{|u^{a}(z)-\chi_n(z)u^{a}(z)|}{2\delta_1}\right)dz= 0$.
	Similarly, as 	$u^{b}\in L_{M^\ast}(\Omega_T)$ there is a constant  $\delta_2>0$ such that 
	$\displaystyle\int_{\Omega_T}M^\ast(z,|u^{b}(z)|/\delta_2)dz <\infty.$ This yields
	$$
	\lim_{n\rightarrow\infty}\int_{\Omega_T}M^\ast\left(z,\frac{|
		u^{b}(z)-\chi_n(z)u^{b}(z)|}{2\delta_2}\right)dz= 0.
	$$	
	Furthermore, having in mind \eqref{chi} we have
	$$
	M^\ast\left(z,\frac{1}{cN\delta_1 n}\left|\sum_{l=1}^Nu^{a}(z)(\partial_{x_l}\chi)\left(\frac{z}{n}\right)\right|\right)
	\leq M^\ast\left(z,\frac{\left|u^{a}(z)\right|}{\delta_1 n} \right)
	\xrightarrow[n\to+\infty]{} 0, \mbox{ a.e. in } \Omega_T
	$$
	and for $n\geq1$, $\displaystyle M^\ast\left(z,\frac{1}{cN\delta_1 n}\left|\sum_{l=1}^Nu^{a}(z)(\partial_{x_l}\chi)\left(\frac{z}{n}\right)\right|\right)\leq
	M^\ast\left(z,\frac{|u^{a}(z)|}{\delta_1}\right)\in L^1(\Omega_T)$. So that the Lebesgue dominated convergence theorem implies
	$$
	\int_{\Omega_T}M^\ast\left(z,\frac{1}{cN\delta_1 n}\left|\sum_{l=1}^Nu^{a}(z)(\partial_{x_l}\chi)\left(\frac{z}{n}\right)\right|\right)dz\xrightarrow[n\to+\infty]{} 0.
	$$
	As regards the two first terms on the right-hand side of the last equality in (\ref{eq_est_un}), we argue similarly as above obtaining
	$$
	\left\|\frac{u}{n}(\partial_t\chi)\left(\frac{z}{n}\right)\right\|_{L^2(\Omega_T)}\xrightarrow[n\to+\infty]{} 0
	\quad \mbox{ and }\quad
	\left\|v-v\chi_n\right\|_{L^2(\Omega_T)}\xrightarrow[n\to+\infty]{} 0.
	$$
	\par\noindent\textbf{Step2.} By virtue of Step 1. we can suppose that the function
	$u\in\mathbf{W}\left(\Omega_T\right)$ is compactly supported in $\overline{\Omega}_T$. Going back to (\ref{seq_spt2}), we shall  prove that $\partial_{t}w$ converges modularly to $\partial_t u$ in $W^{-1,x}L_{M^{\ast}}(\Omega_T)+L^2(\Omega_T)$. Indeed let $\{\psi_{j,i}\}$ be the partition of unity given in the proof of Theorem \ref{th_appr_spt}, and let
	$$
	u=\sum_{j=0}^2\sum_{i=0}^p u_{j,i}:=\sum_{j=0}^2\sum_{i=0}^p \psi_{j,i} u.
	$$
	Accordingly,
	$$
	\partial_t u(z)
	= \displaystyle\sum_{j=0}^2\sum_{i=0}^p\frac{\partial (\psi_{j,i}(z)u(z))}{\partial t}
	=\displaystyle\sum_{j=0}^2\sum_{i=0}^p\left(\psi_{j,i}(z)\partial_tu(z)+u(z)\partial_t\psi_{j,i}(z)\right).
	$$
	Thus 
	$$
	\begin{array}{lll}\displaystyle
		\partial_t u(z)&=\displaystyle\sum_{j=0}^2\sum_{i=0}^p\Big(\left(\sum_{l=1}^N\partial_{x_l}u^{a}(z)+u^{b}(z)+v(z)\right)\psi_{j,i}(z)+u(z)\partial_t\psi_{j,i}(z)\Big)\\&
		=\displaystyle\sum_{j=0}^2\sum_{i=0}^p\Bigg(\sum_{l=1}^N\partial_{x_l}\left(\psi_{j,i}(z)u^{a}(z)\right)
		- \sum_{l=1}^Nu^{a}(z)\partial_{x_l}\psi_{j,i}(z)\\&
		\displaystyle+u^{b}(z)\psi_{j,i}(z)+v(z)\psi_{j,i}(z)+u(z)\partial_t\psi_{j,i}(z)\Bigg).
	\end{array}
	$$
	As regards the time derivative of the function $w$ we have 
	$$
	\begin{array}{ll}
		\partial_t w (z)&= \sum_{j=0}^2\sum_{i=0}^p \partial_t \left(\Big(u_{j,i}^{\varepsilon_{j,i}}\Big)_{(r_{j},-r_{i})}\right)(z)\\&
		= \sum_{j=0}^2\sum_{i=0}^p J_{\varepsilon_{j,i}}\ast  \left(\partial_t\left(\Big(\psi_{j,i}(z)u(z)\Big)_{(r_{j},-r_{i})}\right)\right)\\
		&= \sum_{j=0}^2\sum_{i=0}^p J_{\varepsilon_{j,i}}\ast \Big(\Big(u(z)\partial_t\psi_{j,i}(z)\Big)_{(r_{j},-r_{i})}+
		\Big(\psi_{j,i}(z)\partial_tu(z)\Big)_{(r_{j},-r_{i})}\Big).
	\end{array}
	$$
	Then since $\partial_{t}u=\sum_{l=1}^N\partial_{x_l}u^{a} +u^{b}+v$ we can write 
	$$		
	\begin{array}{lll}
		\partial_t w (z)= \sum_{j=0}^2\sum_{i=0}^pJ_{\varepsilon_{j,i}}\ast \Bigg(\Big(u(z)\partial_t\psi_{j,i}(z)\Big)_{(r_{j},-r_{i})}\\+\Bigg(\psi_{j,i}(z)\Big(\sum_{l=1}^N\partial_{x_l}u^{a}(z)+u^{b}(z)+v(z)\Big)\Bigg)_{(r_{j},-r_{i})}\Bigg)\\
		= \sum_{j=0}^2\sum_{i=0}^pJ_{\varepsilon_{j,i}}\ast \Bigg(\Big(u(z)\partial_t\psi_{j,i}(z)\Big)_{(r_{j},-r_{i})}+\Bigg(\sum_{l=1}^N\partial_{x_l}(\psi_{j,i}(z)u^{a}(z))\\	\displaystyle-\sum_{l=1}^Nu^{a}(z)\partial_{x_l}\psi_{j,i}(z)+u^{b}(z)\psi_{j,i}(z)+v(z)\psi_{j,i}(z)\Bigg)_{(r_{j},-r_{i})}\Bigg).
	\end{array}
	$$		
	Moreover we have 
	$$		
	\begin{array}{lll}
		\partial_t w (z)
		&= \sum_{j=0}^2\sum_{i=0}^p \Bigg( J_{\varepsilon_{j,i}}\ast\Big(u(z)\partial_t\psi_{j,i}(z)\Big)_{(r_{j},-r_{i})}
		+\sum_{l=1}^N\partial_{x_l}\left(J_{\varepsilon_{j,i}}\ast\Big(\psi_{j,i}(z)u^{a}(z)\Big)_{(r_{j},-r_{i})}\right)\\&
		-\sum_{l=1}^NJ_{\varepsilon_{j,i}}\ast\Big(u^{a}(z)\partial_{x_l}\psi_{j,i}(z)\Big)_{(r_{j},-r_{i})}+J_{\varepsilon_{j,i}}\ast \Big(u^{b}(z)\psi_{j,i}(z)\Big)_{(r_{j},-r_{i})}
		\\&+J_{\varepsilon_{j,i}}\ast\Big(v(z)\psi_{j,i}(z)\Big)_{(r_{j},-r_{i})}\Bigg).
	\end{array}
	$$
	Applying  \cite[Theorem 3]{AYGS2017}  to the functions $u^{a},u^{b}\in L_{M^\ast}(\Omega_T)$ and $u,v\in L^{2}(\Omega_T)$, we can deduce that there exists a constant $\lambda>0$ such that for all $\eta>0$, there exist $\tilde{\varepsilon}_{0}>0$ and $\tilde{r}_{0}>0$ such that for all $\varepsilon_{j,i}\leq\tilde{\varepsilon}_{0}$ and all $r_{j},r_{i}\leq\tilde{r}_{0}$ the following estimations hold
	$$	\begin{array}{cc}\displaystyle
		\int_{\Omega_T}M^{\ast}\left(z,\frac{\Big|J_{\varepsilon_{j,i}}\ast\left(\psi_{j,i}(z)u^{a}(z)\right)_{(r_{j},-r_{i})}-\psi_{j,i}(z)u^{a}(z)\Big|}{\lambda}\right)dz\leq\eta,
		\\
		\displaystyle	\int_{\Omega_T}M^{\ast}\left(z,\frac{\Big|J_{\varepsilon_{j,i}}\ast\left(\psi_{j,i}(z)u^{b}(z)\right)_{(r_{j},-r_{i})}-\psi_{j,i}(z)u^{b}(z)\Big|}{\lambda}\right)dz\leq\eta,
		\\
		\displaystyle	\int_{\Omega_T}M^{\ast}\left(z,\frac{\Big|u^{a}(z)\partial_{x_l}\psi_{j,i}(z)-J_{\varepsilon_{j,i}}\ast\left(u^{a}(z)\partial_{x_l}\psi_{j,i}(z)\right)_{(r_{j},-r_{i})}\Big|}{\lambda}\right)dz\leq\eta,
	\end{array}
	$$
	for every $l=1,\cdots,N$ and 
	$$
	\begin{array}{ll}
		\left\|u\partial_t \psi_{j,i}- J_{\varepsilon_{i}}\ast\left(u\partial_t \psi_{j,i}\right)_{(r_{j},-r_{i})}\right\|_{L^2(\Omega_T)}\leq\eta,\quad
		\left\|v \psi_{j,i}- J_{\varepsilon_{i}}\ast\left(v \psi_{j,i}\right)_{(r_{j},-r_{i})}\right\|_{L^2(\Omega_T)}\leq\eta.
	\end{array}
	$$
	Hence,  $\partial_t w$ converges modularly to $\partial_t u$ in $W^{-1,x}L_{M^{\ast}}(\Omega_T)+L^2(\Omega_T)$ and in view of (\ref{est_1}) we conclude that the sequence $w$ converges modularly to $u$ in $ \mathbf{W}\left(\Omega_T\right)$.
\end{proof}


\bibliographystyle{plain}
\bibliography{youssfi_approx}

\def\ocirc#1{\ifmmode\setbox0=\hbox{$#1$}\dimen0=\ht0\advance\dimen0
  by1pt\rlap{\hbox
  to\wd0{\hss\raise\dimen0\hbox{\hskip.2em$\scriptscriptstyle\circ$}\hss}}#1\else{\accent"17
  #1}\fi} \def\cprime{$'$}
\begin{thebibliography}{10}

\bibitem{AF}
R.~A. Adams and J.~F. Fournier.
\newblock {\em Sobolev spaces}, volume 140 of {\em Pure and Applied Mathematics
  (Amsterdam)}.
\newblock Elsevier/Academic Press, Amsterdam, second edition, 2003.

\bibitem{ahm_thesis}
Y.~Ahmida.
\newblock {\em On some functional analysis results in {M}usielak spaces and
  applications to PDEs}.
\newblock Ph.d. thesis, University of {S}idi {M}ohamed {B}en {A}bdellah, 2019.

\bibitem{AYGS2017}
Y.~Ahmida, I.~Chlebicka, P.~Gwiazda, and A.~Youssfi.
\newblock Gossez's approximation theorems in {M}usielak--{O}rlicz--{S}obolev
  spaces.
\newblock {\em J. Funct. Anal.}, 275(9):2538--2571, 2018.

\bibitem{AFY_ms19}
Y.~Ahmida, A.~Fiorenza, and A.~Youssfi.
\newblock {$H=W$} {M}usielak spaces framework.
\newblock {\em Atti Accad. Naz. Lincei Rend. Lincei Mat. Appl.},
  31(2):447--464, 2020.

\bibitem{AY2018.art3}
Y.~Ahmida and A.~Youssfi.
\newblock Poincar\'e-type inequalities in {M}usielak spaces.
\newblock {\em Ann. Acad. Sci. Fenn. Math.}, 44:1041--1054, 2019.

\bibitem{AY_var19}
Y.~Ahmida and A.~Youssfi.
\newblock Variational nonlinear elliptic equations in nonreflexive {M}usielak
  spaces.
\newblock {\em J. Math. Anal. Appl.}, 491(2):124387, 2020.

\bibitem{AS07}
S.~Antontsev and S.~Shmarev.
\newblock Parabolic equations with anisotropic nonstandard growth conditions.
\newblock In {\em Free boundary problems}, volume 154 of {\em Internat. Ser.
  Numer. Math.}, pages 33--44. Birkh\"{a}user, Basel, 2007.

\bibitem{AS09}
S.~Antontsev and S.~Shmarev.
\newblock Anisotropic parabolic equations with variable nonlinearity.
\newblock {\em Publ. Mat.}, 53(2):355--399, 2009.

\bibitem{ALY2016}
E.~Azroul, B.~Lahmi, and A.~Youssfi.
\newblock Strongly nonlinear variational parabolic equations with
  {$p(x)$}-growth.
\newblock {\em Acta Math. Sci. Ser. B (Engl. Ed.)}, 36(5):1383--1404, 2016.

\bibitem{min-double-reg1}
M.~Colombo and G.~Mingione.
\newblock Regularity for double phase variational problems.
\newblock {\em Arch. Ration. Mech. Anal.}, 215(2):443--496, 2015.

\bibitem{bookCF}
D.~Cruz-Uribe and A.~Fiorenza.
\newblock {\em Variable {L}ebesgue spaces}.
\newblock Applied and Numerical Harmonic Analysis. Birkh\"auser/Springer,
  Heidelberg, 2013.
\newblock Foundations and harmonic analysis.

\bibitem{DHHR}
L.~Diening, P.~Harjulehto, P.~H{\"a}st{\"o}, and M.~R{\r{u}}{\v{z}}i{\v{c}}ka.
\newblock {\em Lebesgue and {S}obolev spaces with variable exponents}, volume
  2017 of {\em Lecture Notes in Mathematics}.
\newblock Springer, Heidelberg, 2011.

\bibitem{DNR12}
L.~Diening, P.~N\"{a}gele, and M.~R\r{u}\v{z}i\v{c}ka.
\newblock Monotone operator theory for unsteady problems in variable exponent
  spaces.
\newblock {\em Complex Var. Elliptic Equ.}, 57(11):1209--1231, 2012.

\bibitem{Donaldson74}
T.~Donaldson.
\newblock Inhomogeneous {O}rlicz-{S}obolev spaces and nonlinear parabolic
  initial value problems.
\newblock {\em J. Differential Equations}, 16:201--256, 1974.

\bibitem{EM05}
A.~Elmahi and D.~Meskine.
\newblock Parabolic equations in {O}rlicz spaces.
\newblock {\em J. London Math. Soc. (2)}, 72(2):410--428, 2005.

\bibitem{MR2101516}
A.~Elmahi and D.~Meskine.
\newblock Strongly nonlinear parabolic equations with natural growth terms in
  {O}rlicz spaces.
\newblock {\em Nonlinear Anal.}, 60(1):1--35, 2005.

\bibitem{Gossez82}
J.-P. Gossez.
\newblock Some approximation properties in {O}rlicz-{S}obolev spaces.
\newblock {\em Studia Math.}, 74(1):17--24, 1982.

\bibitem{Gossez1987}
J.-P. Gossez and V.~Mustonen.
\newblock Variational inequalities in {O}rlicz-{S}obolev spaces.
\newblock {\em Nonlinear Anal.}, 11(3):379--392, 1987.

\bibitem{Kaminska15}
A.~Kami\'nska and D.~Kubiak.
\newblock The {D}augavet property in the {M}usielak-{O}rlicz spaces.
\newblock {\em J. Math. Anal. Appl.}, 427(2):873--898, 2015.

\bibitem{book_KJF}
A.~Kufner, O.~John, and S.~Fu\v{c}\'\i{}k.
\newblock {\em Function spaces}.
\newblock Noordhoff International Publishing, Leyden; Academia, Prague, 1977.
\newblock Monographs and Textbooks on Mechanics of Solids and Fluids;
  Mechanics: Analysis.

\bibitem{Lions69}
J.-L. Lions.
\newblock {\em Quelques m\'ethodes de r\'esolution des probl\`emes aux limites
  non lin\'eaires}.
\newblock Dunod; Gauthier-Villars, Paris, 1969.

\bibitem{book_musielak}
J.~Musielak.
\newblock {\em Orlicz spaces and modular spaces}, volume 1034 of {\em Lecture
  Notes in Mathematics}.
\newblock Springer-Verlag, Berlin, 1983.

\bibitem{nagele_thesis}
P.~N{\"a}gele.
\newblock Monotone operators in spaces with variable exponents.
\newblock {\em Diploma thesis}, 2009.

\bibitem{AY2018}
A.~Youssfi and Y.~Ahmida.
\newblock Some approximation results in {M}usielak-{O}rlicz spaces.
\newblock {\em Czechoslovak Math. J.}, 70(145)(2):453--471, 2020.

\end{thebibliography}

\end{document}